%% file: ms.tex
\documentclass[runningheads]{lncse}

\newif\ifdevelopmentVersion
\developmentVersiontrue

\input{./praeambels_and_definitions/artikelpraeambel.tex}

\input{./praeambels_and_definitions/artikel_theoreme_und_farbe.tex}

\input{./praeambels_and_definitions/L-Schema_Paper_makros.tex}

\newif\ifoptionaltext
\newcommand{\optionaltextcolor}{\color{gray}}

\newif\ifcomments
\commentstrue 

\begin{document}

\title{A linear domain decomposition method for two-phase flow in porous media}

\titlerunning{L-scheme for two phase flow}

\author{ David Seus\inst{1} \and Florin A. Radu\inst{2} \and Christian Rohde\inst{1}   }

\authorrunning{Seus et al.}   

\institute{
University of Stuttgart, Institute of Applied Analysis and Numerical Simulation, Pfaffenwaldring 57, 70569  Stuttgart, Germany, {\tt david.seus@ians.uni-stuttgart.de}
\and University of Bergen, Department of Mathematics, Postboks 7803, 5020 Bergen, Norway, {\tt Florin.Radu@uib.no}
}

\maketitle

\begin{abstract}
    This article is a follow up of our submitted  paper \cite{seus_mini27_SeusMitra2017} in which a
    decomposition of the Richards equation along two soil layers was discussed. 
    A decomposed problem was formulated and a decoupling and linearisation technique was presented to solve the problem in each time step in a fixed point type iteration. 
    This article extends these ideas to the case of two-phase in porous media and the convergence of the proposed domain decomposition method is rigorously shown.
\end{abstract}

\section{Introduction}
Soil remediation, enhanced oil recovery, $CO_2$ storage and geothermal energy are among the most important applications of porous media research and are notable examples of multiphase flow processes through porous media. 
In these situations mathematical modelling and simulation are among the most important tools available to predict subsurface processes and to asses feasibility and risk of envisioned technology, since measurements below surface are very difficult, expensive or not possible at all. 
Specifically when considering layered soil with very different porosity and permeability in each layer, the mathematical and computational problems appearing are most challenging as soil parameters may be even discontinuous and the appearing coupled nonlinear partial differential equations change type and degenerate. 
In these settings Newton based solvers can struggle with robustness and convergence.

To overcome the difficulties in robustness, the L-type linearisation, which replaces the Newton solver by a fixed point type iteration, has been proposed and tested in various model settings. 
We refer to \cite{seus_mini27_List2016,seus_mini27_SeusMitra2017} for an overview over the application of the L-scheme to the Richards/Richardson equation as well as comparisons to other methods 
and only mention \cite{seus_mini27_Slodicka2002} and \cite{seus_mini27_Pop2004365}, where the L-scheme was used in combination with mixed finite elements. More recently, the same ideas have been extended to two-phase flow in porous media, c.f. \cite{seus_mini27_Radu2015134} for finite volumes and \cite{seus_mini27_Radu2017} for  mixed finite elements. 
While most of the mentioned papers assume a Lipschitz continuous dependency of the water saturation on the pressure, \cite{seus_mini27_Radu2017} is the first to give error estimates for 
the Hölder continuous case, which is highly relevant due to van Genuchten-Mualem parametrisations falling into this category. The L-scheme has also been applied to other models and coupling problems.
\cite{seus_mini27_Karpinski2017} analyses the method for the case of two-phase flow including dynamic capillary pressure effects.
\cite{seus_mini27_Borregales2017,seus_mini27_Both2017101} propose an optimised Fixed Stress Splitting method, based on an L-type linearisation technique to solve robustly a coupling of flow and geomechanics, modelled by linearised Biot's equation.

The added robustness that L-type linearisations offer, come at the price of slower, i.e. linear convergence. 
Aside from using the L-scheme merely as a preconditioner as mentioned above, another way of optimising convergence speed is by combining the L-scheme with a model based domain decomposition ansatz. 
The physical situation under consideration, in our case, layered soil can be taken into account and a domain decomposition with respect to this physical situations can be performed. 
In \cite{seus_mini27_Berninger2015,seus_mini27_Berninger2010}, the authors considered a substructuring of  the Richards equation along the soil layers and apply monotone multigrid methods to solve the substructured problems. 

(Optimized) Schwarz-Waveform methods for Richards equation were considered in \cite{seus_mini27_Ahmed2017}, where also a posteriori error estimates and stopping criteria were discussed. For the full two-phase flow system domain decomposition based on mortar finite elements and Newton based solvers have been considered in \cite{seus_mini27_Yotov1998,seus_mini27_Yotov1997,seus_mini27_Yotov2001}.

In contrast to the existing approaches, we propose a new domain decomposition solver scheme, independently of the concrete space discretisation for two-phase flow in porous media. The scheme avoids the use of Newton based iterations. Maintaining the form of the equations in physical variables makes the method particularly accessible for application in the engineering context. 
In section \ref{seus_mini27:Problem_description}, we introduce the problem formulation, notation and formulate the iterative scheme. The reader is invited to compare the stated to \cite{seus_mini27_SeusMitra2017} as the ideas are analogous and many of the explanations given there carry over directly to the present case. 
For the sake of brevity, they had to be omitted here. Section \ref{seus_mini27:SectionConvergenceOfScheme} is devoted to the formulation and proof of the main result of this article, the convergence of the scheme. 
We conclude by giving a brief outlook on what questions 
we would like to focus on in the near future.

\section{Problem description}
\label{seus_mini27:Problem_description}
Let $\dom_l \subset \RR^d$ ($l=1,2$) be two Lipschitz domains connected through the interface $\Gamma$.  We consider the flow of two immiscible, incompressible fluids in an isotropic, non-deformable porous medium which is governed by the equations
  \begin{align}
    \Phi_l\partial_t S_l(p_{\nw,l},p_{w,l})
      - \dv\Bigl(\frac{k_{i,l}}{\mu_w} k_{w,l}\bigl( S_l(p_{\nw,l},p_{w,l}) \bigr) \vt{\bm{\nabla}} \bigl(   
	p_{w,l} - z_w\bigr)\Bigr)
      &=  f_{w,l} 
      && \mbox{on } \dom_l\times [0,T] \\
    - \Phi_l\partial_t S_l(p_{\nw,l},p_{w,l})
      - \dv\Bigl(\frac{k_{i,l}}{\mu_g} k_{g,l} \bigl( 1 - S_l(p_{\nw,l},p_{w,l}) \bigr)  \vt{\bm{\nabla}} \bigl( p_{\nw,l} - z_{\nw}\bigr) \Bigr)
      &=  f_{g,l} 
      && \mbox{on } \dom_l \times [0,T] \\
    p_{\alpha,1} = p_{\alpha,2},  \qquad \vec{F_{\alpha,1}}\cdot\vec{n_1} = \vec{F_{\alpha,2}}\cdot\vec{n_1} \qquad \alpha &= w, \nw
    && \mbox{on } \Gamma \times [0,T].
  \end{align}
We adopted a pressure-pressure formulation with wetting and non-wetting pressures $p_{l,w}$, $p_{l,\nw}$ as primary variables, together with the continuity of pressures and fluxes (c.f. \cref{seus_mini27:notation}) as  coupling conditions over the interface. 
Throughout the article, we adhere to the following  notational conventions and abbreviations.

\begin{notation} \label{seus_mini27:notation} $S_l$ is the water saturation and is assumed to be a function of the phase pressures via the capillary pressure saturation relationship $p_{c,l} = p_{\nw,l}- p_{w,l}$. $p_{w,l}$ and $p_{\nw,l}$ are the continuous pressures of the wetting and nonwetting phases on $\dom_l$, respectively. 
$\pwln{n}, \pgln{n}$ denote the pressures at time step $t^n:= n\cdot \tau$ and $S_l^k = \Phi_l S_l(\pgln[l]{k},\pwln[l]{k})$. Here, $\Phi_l$ are the constant porosities  on each $\dom_l$, $\rho_\alpha$ denote the densities of the phases, $\mu_\alpha$ are the viscosities and assuming an intrinsic permeability of the form $\vt{K_l} = k_{i,l}\vt{E_d}$ ($\vt{E_d}$ the identity matrix, dropped in the notation), we abbreviate
\begin{equation}
  \begin{aligned}
  %
    k_{w,l}^k &:= \frac{k_{i,l}}{\mu_w} k_{w,l}\bigl( S_{l}(\pgln[l]{k},\pwln[l]{k})\bigr),
      && k_{g,l}^k :=   \frac{k_{i,l}}{\mu_g}k_{g,l}\bigl(1-S_{l}(\pgln[l]{k},p_{w,l}^k)\bigr), \\
    k_{w,l}^{n,i} &:=  \frac{k_{i,l}}{\mu_w}k_{w,l} \bigl( S_{l}(\pglni{i},\pwlni{i})\bigr),
      && k_{g,l}^{n,i} :=  \frac{k_{i,l} }{\mu_g}k_{g,l}\bigl(1-S_{l}(\pglni{i},\pwlni{i})\bigr),
  \end{aligned}\label{seus_mini27:notation1}
\end{equation}
where $k_{\alpha,l}$ are the relative permeability functions. The pressures $\palni{i}$ are the iterates in our scheme, henceforth called the L-scheme, which will be explained. 
Finally, we write (and already used)
\begin{align} 
    \fluxb[\beta]{w} &:= - \frac{k_{i,l}}{\mu_w} k_{w,l}\bigl( S_l(\pgln[l]{\beta},\pwln[l]{\beta}) \bigr) \vt{\bm{\nabla}} \bigl(   
  \pwln[l]{\beta} - z_w\bigr), \quad
    \fluxb[\beta]{\nw} := - \frac{k_{i,l}}{\mu_g} k_{g,l}\bigl( 1 - S_l(\pgln[l]{\beta},\pwln[l]{\beta}) \bigr)  \vt{\bm{\nabla}} \bigl( \pgln[l]{\beta} - z_{\nw}\bigr) 
\end{align}
for the fluxes, where we abbreviated $z_\alpha = \rho_\alpha g x_3$ for the gravitational term. $\beta$ can be empty, meaning the continuous case, as well as $\beta = k$, meaning the pressure iterate at time step $k$ or  $n,i$, then denoting the $i$-th iteration of the L-scheme. In the latter case, we define 
  $  \fluxb[n,i]{\alpha} := -k_{\alpha,l}^{n,i-1} \vt{\bm{\nabla}} \bigl(   
    \palni[l]{i} - z_\alpha\bigr)$.
  For later use we also define $S_l^{n,i} := \Phi_l S_{l}(\pglni{i},\pwlni{i})$.
  
  Furthermore, denoting the whole domain by $\dom := \dom_1 \cup \Gamma \cup \dom_2 $,  the following spaces will be used. 
  $L^2(\dom)$ is the space of Lebesgue
  measurable,  square integrable  functions over $\dom$.  $H^1(\dom)$ contains functions in $L^2(\dom)$ having also weak derivatives in $L^2(\dom)$.
  $H_0^1(\dom) = \completion{C_0^\infty(\dom)}{H^1}$, where the completion is with respect to the standard
  $H^1$ norm and $C_0^\infty(\dom)$ is the space of smooth functions with compact support in
  $\dom$. The definition for $H^1(\dom_l)$ ($l = 1, 2$) is similar. With $\Gamma$ being a $(d-1)$ dimensional manifold in $\bar{\dom}$, $H^{\frac{1}{2}}(\Gamma)$ contains the traces of $H^1$ functions on $\Gamma$ 
  Given $u \in H^1(\dom)$, by its trace on $\Gamma$ is denoted by $\restrictTo{u}{\Gamma}$. We abbreviate
  \begin{align}
    \Fs_l&:= 
    \left\{ u \in H^1(\dom_l)\, \bigl|\bigr. \, \restrictTo{u}{\oB{l}}
    \equiv 0 \right\}, \\%
    %
    \Fs   &:= \left\{ (u_1,u_2) \in \Fs_1\times \Fs_2 \, \bigl|\bigr. \, {u_1}_{|_\Gamma} \equiv {u_2}_{|_\Gamma}
    \right\},\\
    \Tracespace &\hspace{3.1pt}= \bigl\{ \nu \in H^{1/2}(\Gamma) \,\bigl|\bigr.\, \nu = \onGamma{w} \mbox{ for a } w \in
    H_0^1(\dom) \bigr\}.
  \end{align}
  Note, that $\Fs = H^1_0(\dom)$. 
  $\Tracespace'$ denotes the dual space of $\Tracespace$. $\langle \cdot, \cdot \rangle_X$ will denote the $L^2(X)$ scalar product, with $X$ being one of the sets $\dom$, $\dom_l$ ($l = 1, 2$) or $\Gamma$. Whenever self understood, the notation of the domain of integration $X$ will be dropped. Furthermore, $\spl \cdot , \cdot \spr_\Gamma$ stands also for the duality pairing between $\Tracespace'$ and $\Tracespace$.
\end{notation}

After a backward Euler discretisation in time with time step $\tau  := \tfrac{T}{N}$ for some $N \in \mathbb{N}_0$, the coupled two-phase flow problem in weak form reads

\begin{problem}[Semi-discrete coupled two-phase flow system]
\label{seus_mini27:WeakSemiDiscreteFormulation} Find $(\paln[1]{n},\paln[2]{n})$, $\alpha \in \{w, \nw\}$, such that $\fluxn[l]{\alpha}\cdot \vt{n_l} \in \Tracespace$ and
  \begin{align}
    \spl S_l^{n} - S_l^{n-1},\testfunc{w}\spr  %
      - \tau\spl \fluxn{w},\bm{\nabla}\testfunc{w} \spr 
      + \tau \spl \fluxn[3-l]{w}\cdot \vt{n_l},\onGamma{\testfunc{w}}\spr_\Gamma 
      &=  \tau\spl  f_{w,l}^n ,\testfunc{w} \spr \label{seus_mini27:limitsystemeqn1w} \\ 
    -\spl S_l^{n} - S_l^{n-1},\testfunc{g}\spr  %
      - \tau\spl \fluxn{\nw},\bm{\nabla}\testfunc{g} \spr 
      + \tau \spl \fluxn[3-l]{\nw}\cdot \vt{n_l},\onGamma{\testfunc{g}}\spr_\Gamma 
      &=  \tau\spl  f_{g,l}^n ,\testfunc{g} \spr \label{seus_mini27:limitsystemeqn1g}
  \end{align}
  are satisfied for all $(\testfunc{w},\testfunc{\nw}) \in \Fs$.
\end{problem}
Note, that the pressure coupling is implicitly contained in the weak form, c.f. \cite{seus_mini27_SeusMitra2017}.
The following general assumptions will be used throughout the rest of the article.
\begin{assumptions} \label{seus_mini27:assumptionsondata} For $l=1,2$ we assume that%
\footnote{similar assumptions are used in the literature, c.f. \cite{seus_mini27_Pop2004365}, although more recently, the case of Hölder continuity has been treated, see \cite{seus_mini27_Radu2017}.} %
 \begin{enumerate}[itemsep=0.5ex,label=\alph*)]
  \item the relative permeabilities of the wetting phases $k_{w,l}:[0,1] \rightarrow [0,1]$ are strictly mo\-notonically \emph{increasing} and 
  Lipschitz continuous functions with Lipschitz constants $L_{k_{w,l}}$.
  The relative permeabilities of the non-wetting phases $k_{\nw,l}:[0,1] \rightarrow [0,1]$ are strictly mo\-notonically \emph{decreasing} and 
  Lipschitz continuous functions with Lipschitz constants $L_{k_{\nw,l}}$.
  \item there exists  $m \in \RR$ such that $ \frac{k_{i,l} k_{\alpha,l} }{\mu_\alpha} \geq  m > 0 $, for $\alpha = w, \nw$.
  \item the water saturations $S_l$ are functions of the pressures and the capillary pressure saturation relationships $p_c^l(S_l) := p_{\nw,l} - p_{w,l} $ are monotonically decreasing functions. Therefore the saturations,  $S_l\bigl(p_c^l\bigr) = S_l\bigl(p_{\nw,l} - p_{w,l}\bigr)$ are also monotonically decreasing as functions of $p_c^l$ and moreover assumed to be Lipschitz continuous with Lipschitz constants $L_{S_l}$. 
 \end{enumerate}
\end{assumptions}
  Note, that by abuse of notation, we actually denote by $L_{k_{\alpha,l}}$ the Lipschitz constant of the function $\frac{k_{i,l} k_{\alpha,l} }{\mu_\alpha}$ in (\ref{seus_mini27:notation1}).
%
%
Completely analogous to \cite{seus_mini27_SeusMitra2017}, we introduce an iteration scheme to solve {\itshape Problem \ref{seus_mini27:WeakSemiDiscreteFormulation}} that linearises and decouples simultaneously. 

\begin{problem}[L-scheme] \label{seus_mini27:WeakIterScheme}
Let $\lambda_{\alpha} \in(0,\infty)$ and assume that $\bigl(\paln[1]{n-1} , \paln[2]{n-1}\bigr) \in \Fs$ is given for $\alpha \in \{w,\nw \}$. %
Set $\palni{0}:= p_{\alpha,l}^{n-1}$ as well as
  $  \gali{0} := \fluxb[n-1]{\alpha} \cdot \vt{n_l} - \lambda_{\alpha} \onGamma{\paln[l]{n-1} }$
and assume that for some $i\in\NN$ the approximations $\bigl\{\palni{k}\bigr\}_{k=0}^{i-1}$ 
as well as $\bigl\{\gali{k}\bigr\}_{k=0}^{i-1}$ are already known for $l=1,2, \alpha=\nw,w$. 
Find $\bigl(\palni[1]{i},\palni[2]{i}\bigr) \in \Fs$ 
such that 
\begin{align}
  L_{\alpha,l}&\spl \palni{i},\testfunc{\alpha} \spr 
    - \tau\spl \fluxb[n,i]{\alpha},\bm{\nabla}\testfunc{\alpha} \spr 
    + \tau \spl \lambda_\alpha\onGamma{\palni{i}}  
    + \gali{i} ,\onGamma{\testfunc{\alpha}}\spr_\Gamma  \nonumber \\
  &= L_{\alpha,l}\spl \palni{i-1},\testfunc{\alpha}\spr 
      + (-1)^{\delta_{\alpha,w}} \spl S_l^{n,i-1} - S_l^{n-1},\testfunc{\alpha}\spr	
    + \tau\spl  f_{\alpha,l}^n ,\testfunc{\alpha} \spr
      \label{seus_mini27:weakiterschemeeqna}	\\
  \mbox{ with }\quad &  \spl \gali{i},\onGamma{\testfunc{\alpha}}\spr_{\Gamma}  :=\spl -2\lambda_{\alpha} 
  \onGamma{\palni[3-l]{i-1}}-   \gali[3-l]{i-1},\onGamma{\testfunc{\alpha}}\spr_{\Gamma}, 	\qquad\alpha=w,g \label{seus_mini27:weakgliupdate}
\end{align}
     is fulfilled for all $(\testfunc[1]{\alpha},\testfunc[2]{\alpha})\in \Fs$, where $L_{\alpha,l}>0$, $l=1,2$.
\end{problem}
%
By taking the formal limit in {\itshape Problem \ref{seus_mini27:WeakIterScheme}}, assuming that $\palni{i} \rightarrow \paln{n}$ and $\gali{i} \rightarrow \gal$, for some function $\gal$, the limit system of the L-scheme is
\begin{align}
  (-1)^{\delta_{\alpha,w}+1}\spl S_l^{n}& - S_l^{n-1},\testfunc{\alpha}\spr  %
    - \tau\spl \fluxb[n]{\alpha},\bm{\nabla}\testfunc{\alpha} \spr 
    + \tau \spl \lambda_\alpha \onGamma{\paln{n}}
    + \gal,\onGamma{\testfunc{\alpha}}\spr_\Gamma 
    =  \tau\spl  f_{\alpha,l}^n ,\testfunc{\alpha} \spr \tag{\ref{seus_mini27:weakiterschemeeqna}'} \label{seus_mini27:limititerschemeeqna} \\
  &\mbox{ where } \hspace{1cm}   \spl \gal,\onGamma{\testfunc{\alpha}}\spr_{\Gamma}  :=\spl -2\lambda_{\alpha}
  \onGamma{\paln[3-l]{}} -   \gal[3-l]{},\onGamma{\testfunc{\alpha}}\spr_{\Gamma}.   && \tag{\ref{seus_mini27:weakgliupdate}'}\label{seus_mini27:limitglupdate}
\end{align}
This can be shown to be equivalent to {\itshape Problem \ref{seus_mini27:WeakSemiDiscreteFormulation}} analogously to \cite[Lemma 2]{seus_mini27_SeusMitra2017}. The next section will be devoted to showing, that the L-scheme actually converges to this limit system and make precise the details. 
\section{Convergence of the scheme}
\label{seus_mini27:SectionConvergenceOfScheme}
We are now ready to formulate and prove our main result, the convergence of the L-scheme. 

\begin{theorem} \label{seus_mini27:mainresult}
  Assume there exists a unique solution $(\paln[1]{n},\paln[2]{n})\in \Fs$, $\alpha = \nw,w$, to Problem \ref{seus_mini27:WeakSemiDiscreteFormulation} that 
  additionally fulfills $\sup_{l,\alpha}\norm{\bm{\bm{\nabla}} \bigl( p_{\alpha,l}^n - z_\alpha \bigr) }_{L^\infty} \leq M < \infty $. %
  Let $\lambda_\alpha > 0$ and $L_{\alpha,l}\in \RR $ satisfy 
   $\frac{1}{L_{S_l}}-\sum_{\alpha}\frac{1}{2L_{\alpha,l}} > 0$ for $l=1,2$.
  For arbitrary starting pressures $\palni{0} := \nu_{l,\alpha} \in \Fs_l$, $(l=1,2, \alpha = w,\nw)$, let $\bigl\{\pwlni[1]{i},\pwlni[2]{i}\bigr\}_{i \in \NN_0}, \bigl\{\pglni[1]{i},\pglni[2]{i}\bigr\}_{i \in \NN_0} \in \Fs_l^\NN$ be a sequence of 
  solutions to Problem \ref{seus_mini27:WeakIterScheme}, $\bigl\{\gali{i}\bigr\}_{i \in \NN_0}$ being defined by 
  (\ref{seus_mini27:weakgliupdate}).
  Assume, that the time step $\tau$ satisfies
  \begin{align}
    C(L_{S_l},L_{\alpha,l},M,m):= \frac{1}{L_{S_l}} - \sum_{\alpha}\frac{1}{2L_{\alpha,l}} - \tau\sum_{\alpha} \frac{L_{k_{\alpha,l}}^2 M^2}{2m}   > 0 \label{seus_mini27:timesteprestriction}
  \end{align}
  for $l=1,2$. Then, 
  $\palni{i} \rightarrow \paln{n}$ in $\Fs_l$ and $\gali{i} \rightarrow \gal$ in $\Fs_l'$ as $i\rightarrow \infty$ for 
  $l=1,2$ and both phases.
\end{theorem}

\begin{proof}
 For $\alpha \in \{w,\nw\}$ and $l=1,2$, we introduce the iteration errors $\epali{i}	:= \paln{n}- \palni{i}$ as well as $\egali{i}	:= \gal - \gali{i}$, add 
$L_{\alpha,l}\tspl \paln{n},\testfunc{\alpha}\tspr - L_{\alpha,l}\tspl \paln{n},\testfunc{\alpha}\tspr$ to 
\cref{seus_mini27:limititerschemeeqna} and 
subtract 
\cref{seus_mini27:weakiterschemeeqna} to arrive at
\begin{align}
L_{\alpha,l}\spl& \epali{i},\testfunc{\alpha} \spr %
    + \tau \lambda_\alpha \spl \onGamma{\epali{i}},\onGamma{\testfunc{\alpha}} \spr_\Gamma%
    + \tau   \spl \egali{i},\onGamma{\testfunc{\alpha}} \spr_\Gamma		%
    \nonumber\\
  &+ \tau \Bigl[ \bspl-\fluxb[n]{\alpha}%
 		 {\color{\addnullcolor}-\kalni{i-1}\gradPalnPlusGravity{n} 
 		 +\kalni{i-1}\gradPalnPlusGravity{n}}%
 		 +\fluxai{i},\bm{\nabla}\testfunc{\alpha}\bspr 
 	   \Bigr]			
 	   \label{seus_mini27:mainproofstep1a} \\
  &= L_{\alpha,l}\spl \epali{i-1},\testfunc{\alpha} \spr 
     +\underbrace{(-1)^{\delta_{\alpha,w}} \spl \Sln{n} -\Sln{n-1}, \testfunc{\alpha} \spr %
		-(-1)^{\delta_{\alpha,w}} \spl \Slni{i-1} - \Sln{n-1}, \testfunc{\alpha} \spr}_{%
		(-1)^{\delta_{\alpha,w}} \spl \Sln{n}- \Slni{i-1}, \testfunc{\alpha} \spr}.
 		\nonumber
\end{align}
Inserting $\testfunc{\alpha} := \epali{i}$ in \cref{seus_mini27:mainproofstep1a} 
and noting the identity 
\begin{align}
 L_{\alpha,l}\bspl\epali{i}-\epali{i-1},\epali{i}\bspr = \frac{L_{\alpha,l}}{2}\Bigl[ \bigl\|\epali{i}\bigr\|^2 - 
\bigl\|\epali{i-1}\bigr\|^2 
+ 
\bigl\|\epali{i}-\epali{i-1}\bigr\|^2\Bigr], \label{seus_mini27:mainproofstep1intermediate}
\end{align}
yields 
\begin{align}
  \frac{L_{\alpha,l}}{2}&\Bigl[ \bigl\|\epali{i}\bigr\|^2 %
			  - \bigl\|\epali{i-1}\bigr\|^2 
			  + \bigl\|\epali{i}-\epali{i-1}\bigr\|^2\Bigr] 
    + \tau \lambda_\alpha \spl \onGamma{\epali{i}},\onGamma{\epali{i}} \spr_\Gamma%
  = \spl \Sln{n}- \Slni{i-1}, (-1)^{\delta_{\alpha,w}}\epali{i} \spr  \nonumber \\
      &- \tau  \spl \egali{i},\onGamma{\epali{i}} \spr_\Gamma 	
      -\tau \bspl\Bigl(\kaln{n} -\kalni{i-1}\!\Bigr)\gradPalnPlusGravity{n},\bm{\nabla}\epali{i}\bspr
    -\tau \bspl\kalni{i-1}\bm{\nabla} \epali{i},\bm{\nabla}\epali{i}\bspr.
  \label{seus_mini27:mainproofstep2a}
\end{align}
Summing up \cref{seus_mini27:mainproofstep2a} over $\alpha = w,g$ and adding 
$ \spl \Sln{n}- \Slni{i-1}, \epwli{i-1} -\epgli{i-1} \spr $ 
yields
\begin{align}
  \sum_{\alpha}\frac{L_{\alpha,l}}{2}\Bigl[ \bigl\|\epali{i}\bigr\|^2 %
        &- \bigl\|\epali{i-1}\bigr\|^2 
        + \bigl\|\epali{i}-\epali{i-1}\bigr\|^2\Bigr] 
    {\color{\addnullcolor}+  \underbrace{\spl \Sln{n}- \Slni{i-1}, \epwli{i-1} -\epgli{i-1} \spr}_{I_1}} \nonumber\\
  &= \underbrace{\spl \Sln{n}- \Slni{i-1}, {\color{\addnullcolor} \epwli{i-1}} - \epwli{i} - {\color{\addnullcolor} \bigl(\epgli{i-1}} - \epgli{i}{\color{\addnullcolor}\bigr)} \spr}_{=: I_2} 
  - \tau \sum_{\alpha} \spl \lambda_{\alpha}\onGamma{\epali{i}} + \egali{i},\onGamma{\epali{i}} \spr_\Gamma \nonumber \\  
  &-\underbrace{\tau \sum_{\alpha} \bspl\Bigl(\kaln{n} -\kalni{i-1}\!\Bigr)\gradPalnPlusGravity{n},\bm{\nabla}\epali{i}\bspr}_{=: I_3}
    -\underbrace{\tau \sum_{\alpha} \bspl\kalni{i-1}\bm{\nabla} \epali{i},\bm{\nabla}\epali{i}\bspr.}_{=:I_4} \label{seus_mini27:mainproofstep_summing}
\end{align}

We estimate the assigned terms $I_1$--$I_4$ from (\ref{seus_mini27:mainproofstep_summing}) one by one
and start with $I_1$.
Recall $S_l\bigl(\pwln{},\pgln{}\bigr) = p_c^{-1}\bigl(\pgln{} - \pwln{}\bigr)$ and that $p_c' < 0 $ so that we actually have  the dependence $S_l\bigl(\pwln{},\pgln{}\bigr) = S_l\bigl(\pgln{} - \pwln{}\bigr)$ where $S_l$ is monotonically decreasing. Thereby we have
\begin{align}
 \Bigl| S_l\bigl(\pgln{n} - \pwln{n}\bigr) &- S_l\bigl(\pglni{i-1} - \pwlni{i-1}\bigr) \Bigr|^2
 \leq L_{S_l} \Bigl| S_l\bigl(\pgln{n} - \pwln{n}\bigr) - S_l\bigl(\pglni{i-1} - \pwlni{i-1}\bigr) \Bigr| 
 \Bigl| \epgli{i-1} - \epwli{i-1} \Bigr| \nonumber \\
 &= L_{S_l} \Bigl(S_l\bigl(\pgln{n} - \pwln{n}\bigr) - S_l\bigl(\pglni{i-1} - \pwlni{i-1}\bigr) \Bigr)
 \Bigl( \epwli{i-1} -\epgli{i-1} \Bigr) \label{seus_mini27_estimateForI1}
\end{align}
with the Lipschitz continuity of $S_l$. Therefore, by integrating  (\ref{seus_mini27_estimateForI1}), we estimate $I_1$ by 
\begin{align}
\frac{1}{L_{S_l}} \bigl\| \Sln{n} - \Slni{i-1} \bigr\|^2 %
  \leq \spl \Sln{n} - \Slni{i-1}, \epwli{i-1} - \epgli{i-1} \spr. \label{seus_mini27:I_1}
\end{align}

Young's inequality $\abs{xy}\leq \epsilon \abs{x}^2 + \frac{1}{4\epsilon}\abs{y}^2$, $\epsilon > 0$, applied to the term $I_2$, gives
\begin{align}
  \abs{I_2} = \Bigl| \spl \Sln{n} - \Slni{i-1}, \epwli{i-1} &- \epwli{i} -\bigl(\epgli{i-1} - \epgli{i}\bigr) \spr \Bigr| 
  \leq \frac{L_{w,l}}{2}\bigl\|\epwli{i-1} - \epwli{i}\bigl\|^2   \nonumber\\
  &+ \frac{L_{\nw,l}}{2}\bigl\|\bigl(\epgli{i-1} - \epgli{i}\bigr) \bigr\|^2%
  + \left(\frac{1}{2L_{w,l}} + \frac{1}{2L_{\nw,l}}\right)  \bigl\| \Sln{n}- \Slni{i-1}\bigr\|^2, \label{seus_mini27:I_2}
\end{align}
where we chose $\epsilon_{\alpha} = \frac{L_{\alpha,l}}{2}$ for $\alpha = w, \nw$.

For $I_3$, consider the estimation of the summands
\begin{align}
  \Bigl|\bspl\Bigl(\kaln{n} &-\kalni{i-1}\!\Bigr)\gradPalnPlusGravity{n},\bm{\nabla}\epali{i}\bspr \Bigr|
  \leq \bigl\|\bigl(\kaln{n} -\kalni{i-1}\bigr)\gradPalnPlusGravity{n}\bigr\| 
\bigl\|\bm{\nabla}\epali{i}\bigr\| 				\nonumber \\
  &\leq L_{k_{\alpha,l}}M \bigl\|\Sln{n}-\Slni{i-1}\bigr\| \bigl\|\bm{\nabla}\epali{i}\bigr\|
  \leq L_{k_{\alpha,l}}M\epsilon_{\alpha,l}\bigl\|\Sln{n}-\Slni{i-1}\bigr\|^2 + 
  \frac{L_{k_{\alpha,l}}M}{4\epsilon_{\alpha,l}} 
\bigl\|\bm{\nabla}\epali{i}\bigr\|^2.   \label{seus_mini27:mainproofstep2b}
\end{align}
  Here, we used the Lipschitz-continuity of $k_{\alpha,l}$ 
  and the assumption $\sup_{l,\alpha}\norm{\gradPalnPlusGravity{n}}_{\infty} \leq M$. $\epsilon_{\alpha,l}$ will be chosen later. $I_3$ can therefore be estimated as 
\begin{align}
  \abs{I_3} \leq \tau \sum_{\alpha}L_{k_{\alpha,l}}M \epsilon_{\alpha,l} \bigl\| 	\Sln{n}-\Slni{i-1}\bigr\|^2 + 
  \tau \sum_{\alpha} \frac{L_{k_{\alpha,l}}M}{4\epsilon_{\alpha,l}} 
  \bigl\|\bm{\nabla}\epali{i}\bigr\|^2. \label{seus_mini27:I_3}
\end{align}

Finally, by Assumption \ref{seus_mini27:assumptionsondata}b), we estimate $I_4$ by
  $\tau \bspl\kalni{i-1}\bm{\nabla} \epali{i},\bm{\nabla}\epali{i}\bspr > \tau m \bigl\|\bm{\nabla} \epali{i}  \bigr\|^2$.
Using this and the estimates (\ref{seus_mini27:I_1}), (\ref{seus_mini27:I_2}) and (\ref{seus_mini27:I_3}), equation (\ref{seus_mini27:mainproofstep_summing}) becomes 
\begin{align}
  \sum_{\alpha}\frac{L_{\alpha,l}}{2}\Bigl[ \bigl\|\epali{i}\bigr\|^2 %
        &- \bigl\|\epali{i-1}\bigr\|^2\Bigr] 
    + \frac{1}{L_{S_l}} \bigl\| \Sln{n} - \Slni{i-1} \bigr\|^2 
    + \tau \sum_{\alpha} \spl \lambda_{\alpha}\onGamma{\epali{i}} + \egali{i},\onGamma{\epali{i}} \spr_\Gamma \nonumber\\
  &\leq \sum_{\alpha}\Bigl( \frac{1}{2L_{\alpha,l}} + \tau L_{k_{\alpha,l}} M \epsilon_{\alpha,l} \Bigr)\bigl\| \Sln{n}-\Slni{i-1}\bigr\|^2  
  + \tau \sum_{\alpha} \left(\frac{L_{k_{\alpha,l}}M}{4\epsilon_{\alpha,l}} - m \right)  
  \bigl\|\bm{\nabla}\epali{i}\bigr\|^2. \tag{\ref{seus_mini27:mainproofstep_summing}'}\label{seus_mini27:mainproofstep3}
\end{align}

In order to deal with the interface terms $\tau \spl \egali{i},\epaliOnGamma{i} \spr_\Gamma$, recall 
that
$\spl\cdot,\cdot\spr_\Gamma$ denotes both scalar product in $\Tracespace$ and dual pairing for functionals in $\Tracespace'$. Subtracting
$(\ref{seus_mini27:weakgliupdate})$ from $(\ref{seus_mini27:limitglupdate})$, i.e. obtaining %
$ \egali{i} = -2\lambda_{\alpha} \epaliOnGamma[3-l]{i-1} -  \egali[3-l]{i-1}$, 
  we  get
  \begin{align}
    \ifoptionaltext
    { \optionaltextcolor
    \spl \egali{i+1}, \egali{i+1} \spr_\Gamma }&{\optionaltextcolor%
    = 4\lambda_{\alpha}^2 \spl \epaliOnGamma[3-l]{i}, \epaliOnGamma[3-l]{i} \spr_\Gamma %
    + \spl \egali[3-l]{i}, \egali[3-l]{i} 
    \spr_\Gamma %
    + 4\lambda_{\alpha}  \spl \epaliOnGamma[3-l]{i}, \egali[3-l]{i}\spr_\Gamma}, \mbox{ 
  \optionaltextcolor thus }\nonumber \\
    \fi
    \bigl\| \epaliOnGamma[l]{i} \bigr\|^2_\tGamma  &= \frac{1}{  4\lambda_{\alpha}^2 } %
    \left(\bigl\| \egali[3-l]{i+1} \bigr\|^2_\tGamma %
    -\bigl\| \egaliOnGamma[l]{i} \bigr\|^2_\tGamma %
    - 4\lambda_{\alpha}  \spl \epaliOnGamma[l]{i}, \egali[l]{i}\spr_\tGamma \right).
    \label{seus_mini27:boundarytrickpre}
  \end{align}

Inserting eq. (\ref{seus_mini27:boundarytrickpre}) in eq. (\ref{seus_mini27:mainproofstep3}), we arrive at
\begin{align}
    \biggl(\frac{1}{L_{S_l}} - \sum_{\alpha}&\Bigl( \frac{1}{2L_{\alpha,l}} + \tau L_{k_{\alpha,l}} M \epsilon_{\alpha,l} \Bigr)\! \biggr) \bigl\| \Sln{n} - \Slni{i-1} \bigr\|^2 
    +\tau \sum_{\alpha} \left(m - \frac{L_{k_{\alpha,l}}M}{4\epsilon_{\alpha,l}}  \right)  
    \bigl\|\bm{\nabla}\epali{i}\bigr\|^2 
    \nonumber\\
  &\leq \sum_{\alpha}\frac{L_{\alpha,l}}{2}\Bigl[\bigl\|\epali{i-1}\bigr\|^2 - \bigl\|\epali{i}\bigr\|^2 \Bigr] 
  + \tau \sum_{\alpha} \frac{1}{  4\lambda_{\alpha} } %
  \left(\bigl\| \egaliOnGamma[l]{i} \bigr\|^2_\tGamma - \bigl\| \egali[3-l]{i+1} \bigr\|^2_\tGamma \right).
  \label{seus_mini27:mainproofstep4}
\end{align}

Now choose $\epsilon_{\alpha,l} = \frac{L_{k_{\alpha,l}}M}{2m}$ such that $m -\frac{L_{k_{\alpha,l}}M}{4\epsilon_{\alpha,l}} = \frac{m}{2} > 0$ for both $l$ and $\alpha$. 
Taking   into account that by assumption $L_{\alpha,l}$ have been chosen large enough that $\frac{1}{L_{S_l}}-\sum_{\alpha}\frac{1}{2L_{\alpha,l}} 
> 0$ and that (\ref{seus_mini27:timesteprestriction}) holds,
summing (\ref{seus_mini27:mainproofstep4}) over iterations $i=1,\dots,r$  then leads to
\begin{align}
  \sum_{i=1}^r C(L_{S_l},L_{\alpha,l},M,m)& \bigl\| \Sln{n} - \Slni{i-1} \bigr\|^2 
  +\tau \sum_{i=1}^r\sum_{\alpha} \frac{m}{2}  
  \bigl\|\bm{\nabla}\epali{i}\bigr\|^2 
  \nonumber\\
  &\leq \sum_{\alpha}\frac{L_{\alpha,l}}{2}\Bigl[\bigl\|\epali{0}\bigr\|^2 - \bigl\|\epali{r}\bigr\|^2 \Bigr] 
  + \tau \sum_{\alpha} \frac{1}{  4\lambda_{\alpha} } %
  \left(\bigl\| \egaliOnGamma[l]{1} \bigr\|^2_\tGamma - \bigl\| \egali[3-l]{r+1} \bigr\|^2_\tGamma \right), 
  \label{seus_mini27:mainproofstep5}
\end{align}
where the appearing telescopic property of sums on the right hand side have been exploited. This implies the estimates
\begin{align}
  \sum_{i=1}^r\sum_{l=1}^2C(L_{S_l},L_{\alpha,l},M,m)& 
  \bigl\| \Sln{n} - \Slni{i-1} \bigr\|^2  
  \leq  \sum_{l,\alpha}\frac{L_{\alpha,l}}{2}\bigl\|\epali{0}\bigr\|^2
  +\tau \sum_{l,\alpha} \frac{1}{  4\lambda_{\alpha} } %
  \bigl\| \egaliOnGamma[l]{1} \bigr\|^2_\tGamma, \label{seus_mini27:mainproofstep6a} \\%
  \tau \sum_{i=1}^r\frac{m}{2}\bigl\|\bm{\bm{\nabla}} e_p^i \bigr\|^2%
  &\leq  \sum_{l,\alpha}\frac{L_{\alpha,l}}{2}\bigl\|\epali{0}\bigr\|^2
  +\tau \sum_{l,\alpha} \frac{1}{  4\lambda_{\alpha} } %
  \bigl\| \egaliOnGamma[l]{1} \bigr\|^2_\tGamma,
  \label{seus_mini27:mainproofstep6b}
\end{align}
for which we introduced the abbreviation $\bigl\|\bm{\bm{\nabla}} e_p^i \bigr\|^2 := \sum_{\alpha} \bigl\|\bm{\nabla}\epali{i}\bigr\|^2$.
Since the right hand sides are independent of $r$, we thereby conclude that %
$ \bigl\| \Sln{n} - \Slni{i-1} \bigr\| $, $\bigl\|\gradepali{i} \bigr\|$ $\longrightarrow 0$ 	 as $i\rightarrow \infty$. Due to the partial homogeneous Dirichlet boundary, the Poincaré inequality is applicable %
  %
  %
  for functions in $\Fs_l$ so that (\ref{seus_mini27:mainproofstep6b}) further 
  implies $\bigl\|\epali{i} \bigr\|$ $\longrightarrow 0$ 	 as $i\rightarrow \infty$.

  In order to show that $\egali{i} \rightarrow 0$ in $\Fs_l'$, we subtract for both phases again 
  (\ref{seus_mini27:weakiterschemeeqna}) from (\ref{seus_mini27:limititerschemeeqna}) and consider only test functions in $\testfunc{\alpha} 
\in 
  C_0^\infty(\dom_l)$, i.e.
  \begin{align}
    -\tau \bspl\fluxn{\alpha} - \fluxai{i},\bm{\bm{\nabla}}\testfunc{\alpha}\bspr
      &= -L_{\alpha,l}\spl \epali{i},\testfunc{\alpha} \spr %
	+L_{\alpha,l}\spl \epali{i-1},\testfunc{\alpha} \spr 
	+(-1)^{\delta_{\alpha,w}} \spl \Sln{n}- \Slni{i-1}, \testfunc{\alpha} \spr. \label{seus_mini27:giconvergencestep1}
  \end{align}
  Thus, $\dv\bigl(\fluxn{\alpha} - \fluxai{i}\bigr)$ exists in $L^2(\dom_l)$ and
   \begin{align}
      -\tau \dv\bigl(\fluxn{\alpha} - \fluxai{i}\bigr)
      = L_{\alpha,l} \bigl(\epali{i} - \epali{i-1}\bigr)  
      -(-1)^{\delta_{\alpha,w}}\bigl( \Sln{n}- \Slni{i-1}\bigr) \label{seus_mini27:giconvergencestep2}
   \end{align}
  almost everywhere, from which we deduce for $\testfunc{\alpha}$ now taken to be in $\Fs_l$
  \begin{align}
      \Bigl|\bspl\dv\bigl(\fluxn{\alpha} - \fluxai{i}\bigr),\testfunc{\alpha} \bspr\Bigr| 
	&\leq \frac{L_{\alpha,l}}{\tau}\bigl\|\epali{i} - \epali{i-1} \bigr\|\bigl\|\testfunc{\alpha}\bigr\| 
	+ \frac{1}{\tau}\bigl\| \Sln{n}- \Slni{i-1}\bigr\| \bigl\| \testfunc{\alpha} \bigr\|.
      \label{seus_mini27:giconvergencestep3}
  \end{align}
  Introducing the abbreviation $\bigl|\Psi_{\alpha,l}^{n,i}\bigl(\testfunc{\alpha}\bigr)\bigr|$ for the left hand side of (\ref{seus_mini27:giconvergencestep3}),
  \begin{align}
    \sup_{\stackrel{\testfunc{\alpha} \in \Fs_l}{\testfunc{\alpha} \neq 0}} 
    &\frac{\bigl|\Psi_l^{n,i}\bigl(\testfunc{\alpha}\bigr)\bigr|}{\norm{\testfunc{\alpha}}_{\Fs_l}} 
    \leq \frac{L_{\alpha,l}}{\tau}\bigl\|\epali{i} - \epali{i-1} \bigr\|
    + \frac{1}{\tau}\bigl\| \Sln{n}- \Slni{i-1}\bigr\|
    \longrightarrow 0 \quad (i \rightarrow \infty) \label{seus_mini27:giconvergencestep4}
  \end{align}
  follows as a consequence of (\ref{seus_mini27:mainproofstep6b}). In other words $\bigl\|\Psi_l^{n,i} \bigr\|_{\Fs_l'} \rightarrow 0$ as 
$i 
  \rightarrow \infty$.
  On the other hand, starting again from (\ref{seus_mini27:mainproofstep1a}) (without the added zero term), this time however 
inserting   $\testfunc{\alpha} \in \Fs_l$ and integrating again by parts, keeping in mind (\ref{seus_mini27:giconvergencestep2}), 
one notices
  \begin{align}
    \spl \egali{i},\onGamma{\testfunc{\alpha}} \spr_\Gamma   
      = - \lambda_{\alpha} &\spl\epaliOnGamma{i},\onGamma{\testfunc{\alpha}} \spr_\Gamma 
      +\bspl \bigl[ \fluxn{\alpha}- \fluxai{i} \bigr]\cdot\vt{n_l},\onGamma{\testfunc{\alpha}} 
  \bspr_\Gamma . \label{seus_mini27:giconvergencestep5}
  \end{align}
    We already know, that $\bigl\|\epali{i}\bigr\|_{\Fs_l} \rightarrow 0$ as $i \rightarrow 0$ and we will use the 
continuity   of the trace operator to deal with the term $\spl\epaliOnGamma{i},\onGamma{\testfunc{\alpha}} \spr_\Gamma $. 
For the last summand in (\ref{seus_mini27:giconvergencestep5}) we have by the 
integration by parts formula
  \begin{align}
    \bspl \bigr[ \fluxn{\alpha}- \fluxai{i}  \bigr]\cdot\vt{n_l},\onGamma{\testfunc{\alpha}} 
  \bspr_\Gamma 
    &= \Psi_{\alpha,l}^{n,i}(\testfunc{\alpha}) 
      + \bspl \fluxn{\alpha}- \fluxai{i}, \bm{\bm{\nabla}}\testfunc{\alpha}\bspr,
  \label{seus_mini27:giconvergencestep6}
  \end{align}
  and the second term can be estimated by
  \begin{align}
    \Bigl|\bspl \kaln{n} \bm{\bm{\nabla}} \bigl( \paln{n} + z_{\alpha} \bigr)  
    &- \kalni{i-1}\bm{\bm{\nabla}} \bigl( \palni{i} + z_{\alpha} \bigr) 
    ,\bm{\bm{\nabla}}\testfunc{\alpha}\bspr\Bigr| \nonumber \\
    &\leq
    \Bigl|\bspl \bigl( \kaln{n} -\kalni{i-1}\bigr) \bm{\bm{\nabla}} \bigl( \paln{n} + z_{\alpha} \bigr)  
    - \kalni{i-1}\gradepali{i} ,\bm{\bm{\nabla}}\testfunc{\alpha}\bspr\Bigr|
\nonumber   \\%
    &\leq  L_{k_l} M \bigl\|\Sln{n} - \Slni{i-1}\bigr\| \bigl\|\testfunc{\alpha}\bigr\|_{\Fs_l} + 
    M_{k_l}\bigl\|   \gradepali{i}\bigr\| \bigl\|\testfunc{\alpha} \bigr\|_{\Fs_l},
  \end{align}
  where we used the same reasoning as in (\ref{seus_mini27:mainproofstep2b}) and $\max|k_l| \leq M_{k_l}$. 
  With this, we get 
  \begin{align}
    \sup_{\stackrel{\testfunc{\alpha} \in \Fs_l}{\norm{\testfunc{\alpha}}_{\Fs_l} = 1}}&
    \Bigl|\bspl \bigr[ \fluxn{\alpha}- \fluxai{i} \bigr]\cdot\vt{n_l},\testfunc{\alpha}\bspr_\Gamma\Bigr| 
    \leq \bigl\|\Psi_{\alpha,l}^{n,i}\bigr\|_{\Fs_l'}
    %
    + L_{k_l} M \bigl\|\Sln{n} - \Slni{i-1}\bigr)\bigr\|  + M_{k_l}\bigl\| \gradepali{i}\bigr\| 
  \longrightarrow 0  \label{seus_mini27:giconvergencestep7} 
  \end{align}
  as $i \rightarrow \infty $ from (\ref{seus_mini27:giconvergencestep6}). 
  Finally, we deduce from (\ref{seus_mini27:giconvergencestep5}) and the continuity of the trace operator 
  on Lipschitz domains
   \begin{align*}
    \sup_{\stackrel{\testfunc{\alpha} \in \Fs_l}{\testfunc{\alpha} \neq 0}}
    &\frac{\bigl|\spl \egali{i},\onGamma{\testfunc{\alpha}} \spr_\Gamma\bigr|}{\norm{\testfunc{\alpha}}_{\Fs_l}} 
    \leq \lambda_{\alpha}\tilde{C}\bigl\| \epali{i} \bigr\|_{\Fs_l} 
      + \bigl\|\Psi_{\alpha,l}^{n,i}\bigr\|_{\Fs_l'} 
      + L_{k_l} M \bigl\|\Sln{n} - \Slni{i-1}\bigr)\bigr\|  
      + M_{k_l}\bigl\| \gradepali{i} \bigr\| \longrightarrow 0 , 
  \end{align*}
  as $i \rightarrow \infty$.
  This shows $\egali{i} \rightarrow 0$ in $\Fs_l'$ for $l = 1,2$ and $\alpha = w, \nw$ and concludes the proof. \hfill\bewend
\end{proof}

%
%
%


\def\cprime{$'$}

\section{Conclusion}

We proposed and analysed a fully implicit domain decomposition method for efficiently solving two-phase flow in heterogeneous porous media. The developed scheme avoids using the Newton method.
The generalisation to several soil layers, the analysis of a concrete discretisation in space as well as thorough numerical testing are left for future work. 

\section*{Acknowledgments}  

This work was partially supported by the NFR supported project CHI \#25510 and by the VISTA project \#6367.
%
%
%

\ifdevelopmentVersion
    \bibliographystyle{vmams}
    \bibliography{articlebibliography}
\else
  \ifx\undefined\bysame
  \newcommand{\bysame}{\leavevmode\hbox to3em{\hrulefill}\,}
  \fi

\fi

\end{document}

%% file: praeambels_and_definitions/artikelpraeambel.tex
\usepackage{etex} 

\usepackage{ifluatex}
\ifluatex
  \usepackage{lualatex-math}
\else
  \usepackage[utf8]{inputenx}
  \input{ix-utf8enc.dfu}
\fi
\usepackage[dvipsnames,svgnames,table]{xcolor} 

\usepackage[T1]{fontenc} 
\usepackage{latexsym}
\usepackage[full]{textcomp}
\usepackage{eucal}
\usepackage{fixmath}
\usepackage{mathrsfs} 

\usepackage[left = 2.8 cm, right = 2.8 cm, top = 2.8 cm, bottom = 3.5cm]{geometry} 

\usepackage[shortlabels]{enumitem} 

\usepackage{amsfonts}
\usepackage{amssymb}
\usepackage{bbm}
\usepackage{amsmath}
\usepackage{exscale}
\usepackage{amstext}

\usepackage{amsmath}

\usepackage{bm} 
\usepackage{amsbsy}

\ifluatex
\else
  \usepackage{uniinput} 
\fi

\usepackage{float,scrhack}

\usepackage{epsfig}
\usepackage{graphics} 

\ifdevelopmentVersion
    \usepackage{tikz}
    \usepackage{pgfplots}
    \pgfplotsset{compat=newest,%
    }
    \usetikzlibrary{arrows,%
                    arrows.meta,%
                    petri,%
                    topaths,%
                    fit,%
                    positioning,%
                    decorations.pathmorphing,%
                    backgrounds,
                    calc,%
    }%

    \usetikzlibrary{external}
    \usepgfplotslibrary{external}
    \tikzset{external/system call={lualatex \tikzexternalcheckshellescape --shell-escape -halt-on-error 
    -interaction=batchmode -jobname "\image" "\texsource"}}

    \tikzset{
    png export/.style={
        external/system call=%
        {lualatex \tikzexternalcheckshellescape --shell-escape -halt-on-error -interaction=batchmode -jobname 
    "\image" "\texsource" && %
        convert -density 300 "\image.pdf" "\image.png"},
    }
    }

    \tikzset{%
    /pgf/images/external info,
    use png/.style={png export,png import},
    png import/.code={%
        \tikzset{%
        /pgf/images/include external/.code={%
            \includegraphics%
            [width=\pgfexternalwidth,height=\pgfexternalheight]%
            {{##1}.png}%
        }%
        }%
    }%
    }

    \tikzset{png export}

    \tikzifexternalizing{%
    }{%
    }%

\fi

\usepackage{nameref}
\usepackage[english]{varioref}
\usepackage{hyperref}
\hypersetup{
      unicode=true,           
     plainpages=false,		
     linkcolor=blue,          
     citecolor=darkgreen,        
     filecolor=black,      
     urlcolor=blue           
     pdfborder={0 0 1},		
     linkbordercolor=gray!15,
     citebordercolor=green!15,
}
\usepackage{memhfixc}
\usepackage[]{cleveref}

\allowdisplaybreaks[1] 



%% file: praeambels_and_definitions/artikel_theoreme_und_farbe.tex
\definecolor{mediumblue}{RGB}{0,0,205}
\definecolor{navyblue}{RGB}{0,0,128}
\definecolor{midnightblue}{RGB}{25,25,112}
\definecolor{royalblue4}{RGB}{39,64,139}
\definecolor{blue3}{RGB}{0,0,205}
\definecolor{steelblue3}{RGB}{79,148,205}
\definecolor{steelblue4}{RGB}{54,100,139}

\definecolor{brown}{RGB}{165,42,42} 
\definecolor{brown3}{RGB}{205,51,51} 
\definecolor{brown4}{RGB}{139,35,35} 
\definecolor{red3}{RGB}{205,0,0} 
\definecolor{tomato}{RGB}{205,79,57} 
\definecolor{firebrick3}{RGB}{205,38,38}
\definecolor{firebrick4}{RGB}{139,26,26}

\definecolor{gold}{RGB}{255,215,0}
\definecolor{gold3}{RGB}{238,201,0} 
\definecolor{darkgoldenrod1}{RGB}{255,185,15}
\definecolor{goldenrod1}{RGB}{255,193,37}
\definecolor{goldenrod}{RGB}{218,165,32}

\definecolor{orange}{rgb}{.9,.6,.1}
\definecolor{dunkelorange}{rgb}{.9,.5,.0}
\definecolor{orange2}{RGB}{238,154,0}
\definecolor{orange3}{RGB}{205,133,0}

\definecolor{darkgreen}{RGB}{0,100,0}
\definecolor{green3}{RGB}{0,205,0} 
\definecolor{olivedrab}{RGB}{107,142,35}
\definecolor{olivedrab2}{RGB}{179,238,58}
\definecolor{olivedrab3}{RGB}{154,205,50}
\definecolor{forestgreen}{RGB}{34,139,34}
\definecolor{darkolivegreen}{RGB}{85,107,47}
\definecolor{darkolivegreen4}{RGB}{110,139,61}
\definecolor{khaki3}{RGB}{205,198,115}

\definecolor{grey}{rgb}{0.5,0.5,0.5}
\definecolor{dimgrey}{RGB}{105,105,105}
\definecolor{dimgrey2}{RGB}{153,153,153}
\definecolor{dimgrey3}{RGB}{181,181,181}
\definecolor{lightgrey}{RGB}{211,211,211}
\definecolor{lightergrey}{RGB}{201,201,201}
\definecolor{verylightgrey}{RGB}{222,222,222}

\definecolor{braun}{rgb}{.6,.5,.1}
\definecolor{braun2}{rgb}{.6,.4,.1}
\definecolor{tan4}{RGB}{149,90,43}

\definecolor{wheat3}{RGB}{205,186,150}
\definecolor{wheat4}{RGB}{139,126,102}

\newcommand{\bewend}{\rule{1ex}{1ex}}
 
\makeatletter
  \def\thmt@headstyle@margincolored{%
    \makebox[0pt][r]{\color{black!40}\NUMBER\ }\NAME\NOTE
  }
\makeatother

\newtheorem{notation}{Notation}{\bfseries}{\itshape}
\newtheorem{assumptions}{Assumptions}{\bfseries}{\itshape}

%% file: praeambels_and_definitions/L-Schema_Paper_makros.tex
\newcommand{\dom}{\Omega}

\newcommand{\vt}[1]{\boldsymbol{#1}}

\newcommand{\addnullcolor}{black}

\newcommand{\oB}[1]{\partial\Omega_#1 \cap \partial\Omega}

\newcommand{\nw}{\mathfrak{g}}

\newcommand{\paln}[2][l]{p_{\alpha,#1}^{#2}}
\newcommand{\pwln}[2][l]{p_{w,#1}^{#2}}
\newcommand{\pgln}[2][l]{p_{\nw,#1}^{#2}}
 
\newcommand{\palni}[2][l]{p_{\alpha,#1}^{n,#2}}
\newcommand{\pwlni}[2][l]{p_{w,#1}^{n,#2}}
\newcommand{\pglni}[2][l]{p_{\nw,#1}^{n,#2}}

\newcommand{\Sln}[2][l]{S_{#1}^{#2}}
\newcommand{\Slni}[2][l]{S_{#1}^{n,#2}}

\newcommand{\gradPalnPlusGravity}[2][l]{\bm{\nabla} \bigl( \paln[#1]{#2} - z_{\alpha} \bigr)}

\newcommand{\fluxn}[2][l]{\vt{F_{#2,#1}^{n}}}
\newcommand{\fluxb}[2][]{\vt{F_{#2,l}^{#1}}}
\newcommand{\fluxai}[2][l]{\vt{F_{\alpha,#1}^{n,#2}}}

\newcommand{\kalni}[2][l]{k_{\alpha,#1}^{n,#2}}

\newcommand{\kaln}[2][l]{k_{\alpha,#1}^{#2}}

\newcommand{\gali}[2][l]{g_{\alpha,#1}^{#2}}

\newcommand{\gal}[1][l]{g_{\alpha,#1}}

\newcommand{\epali}[2][l]{e_{p,#1}^{\alpha,#2}}
\newcommand{\epwli}[2][l]{e_{p,#1}^{w,#2}}
\newcommand{\epgli}[2][l]{e_{p,#1}^{\nw,#2}}
\newcommand{\gradepali}[2][l]{\bm{\nabla}e_{p,#1}^{\alpha,#2}}

\newcommand{\egali}[2][l]{e_{g,#1}^{\alpha,#2}}

\newcommand{\epaliOnGamma}[2][l]{e_{p,#1}^{\alpha,#2}{}}

\newcommand{\egaliOnGamma}[2][l]{e_{g,#1}^{\alpha,#2}{}}

\newcommand{\testfunc}[2][l]{\varphi_{#2,#1}}

\newcommand{\spl}{\bigl\langle}
\newcommand{\spr}{\bigr\rangle}
\newcommand{\tspl}{\langle}
\newcommand{\tspr}{\rangle}
\newcommand{\bspl}{\Bigl\langle}
\newcommand{\bspr}{\Bigr\rangle}

\newcommand{\Fs}{\mathcal{V}}
\newcommand{\Tracespace}{H^{1/2}_{00}(\Gamma)}
\newcommand{\tGamma}{{\Gamma}}
\newcommand{\onGamma}[1]{{#1}{|_\Gamma}}
\newcommand{\restrictTo}[2]{{#1}{}{|_{#2}}}
\newcommand{\RR}{\mathbb{R}}

\newcommand{\NN}{\mathbb{N}}

\DeclareMathOperator{\dv}{\bm{\nabla}\cdot}

\providecommand{\abs}[1]{\lvert#1\rvert}
\providecommand{\norm}[1]{\lVert#1\rVert}

\newcommand{\completion}[2]{\overline{#1}{}^{#2}}